\documentclass[12pt,a4paper]{article}
\usepackage{amsmath,amssymb,amscd,amsthm,color,array}
\usepackage[truedimen,margin=25truemm]{geometry}

\usepackage{graphicx,mathrsfs,amssymb,amsmath,color,hyperref}
\usepackage{accents}
\usepackage[latin1]{inputenc}
\usepackage[truedimen,margin=25truemm]{geometry}

\definecolor{royalazure}{rgb}{0.0, 0.22, 0.66}

\numberwithin{equation}{section}
\newtheorem{theorem}{\bf Theorem}[section]
\newtheorem{lemma}[theorem]{\bf Lemma}
\newtheorem{proposition}[theorem]{\bf Proposition}
\newtheorem{definition}{\bf Definition}[section]
\newtheorem{corollary}[theorem]{Corollary}
\theoremstyle{remark}

\newtheorem{example}{\bf Example}[section]

\DeclareMathOperator{\op}{op}

\newcommand{\A}{{\mathbb A}}
\newcommand{\B}{{\mathbb B}}

\newcommand{\N}{{\mathbb N}}

\newcommand{\R}{{\mathbb R}}

\newcommand{\Z}{{\mathbb Z}}

\newcommand{\sK}{{\mathscr K}}
\newcommand{\sL}{{\mathscr L}}

\newcommand{\sS}{{\mathscr S}}
\newcommand{\sU}{{\mathscr U}}

\newcommand{\ga}{\alpha}
\newcommand{\gb}{\beta}

\newcommand{\gG}{\Gamma}

\newcommand{\gve}{\varepsilon}

\newcommand{\gl}{\lambda}
\newcommand{\gL}{\Lambda}
\newcommand{\go}{\omega}

\newcommand{\gvp}{\varphi}

\newcommand{\gt}{\theta}

\newcommand{\gs}{\sigma}

\def\skp#1{\langle#1\rangle}
\DeclareMathOperator{\crit}{crit}
\DeclareMathOperator{\graph}{graph}
\DeclareMathOperator{\ind}{ind}

\DeclareMathOperator{\Tr}{Tr}
\DeclareMathOperator{\tr}{tr}

\numberwithin{equation}{section}

\begin{document}
\begin{center}
{\LARGE Operator Algebras  Associated with \\Quantized Canonical Transformations} \vspace{3mm} \\
{\large Anton Savin} \vspace{2mm} \\
{\large Peoples' Friendship University of Russia } \vspace{3mm} \\
and \vspace{3mm} \\
{\large Elmar Schrohe} \vspace{2mm}\\
{\large Leibniz Universit\"at Hannover} 
\end{center}

\medskip
\medskip
\medskip

\begin{minipage}[c]{150truemm}
{\bf Abstract.}\; 
We review an approach to the index theory of operator algebras associated with  Lie groups of quantized canonical transformations. Main points are an ellipticity condition ensuring the Fredholm property, the definition of localized algebraic and analytic indices and the proof of their equality. 
This framework encompasses many well-known index problems, such as the classical theory on closed manifold, the Atiyah-Weinstein  problem and the index theory for operators with shifts. 
\end{minipage}

%\title[Operator algebras  of quantized canonical transformations]{Operator Algebras  associated with quantized canonical transformations}
%
%\author{Anton Savin} 
%\address{Peoples' Friendship University of Russia (RUDN University),  6 Miklukho-Maklaya St, Moscow, 117198, Russia}
%\email{antonsavin@mail.ru}
%
%\author{Elmar Schrohe} 
%\address{Institute of Analysis, Leibniz Universität Hannover, Welfengarten 1, 30167 Hannover, Germany}
%\email{schrohe@math.uni-hannover.de}

%\thanks{The support of the DFG through grant SCHR 319/8-1 is gratefully acknowledged.}

%\begin{abstract}
%We review an approach to the index theory of operator algebras associated with  Lie groups of quantized canonical transformations. Main points are an ellipticity condition ensuring the Fredholm property, the definition of localized algebraic and analytic indices and the proof of their equality. 
%This framework encompasses many well-known index problems, such as the classical theory on closed manifold, the Atiyah-Weinstein  problem and the index theory for operators with shifts. 
%\end{abstract} 
%\maketitle

\medskip
\medskip


\medskip
\medskip
\medskip

\section{Introduction}
Given a closed manifold $M$ and a discrete group $G$ of quantized canonical 
transformations $\Phi_g$, $g\in G$,  we are developing an index theory for the algebra of operators of the form 
$$D = \sum_g D_g\Phi_g+K$$ 
on $L^2(M)$. Here, the $D_g$ are zero order pseudodifferential operators on $M$, only finitely many different from zero, and $K$ is in the ideal $\sK(L^2(M))$ of compact operators. 

We present suitable notions of symbols and ellipticity and show the Fredholm property of elliptic elements. 
As a first step towards an index formula, we then focus on the case where $G$ is a finite extension of $\Z^d$, $d\in \N_0$. 
We introduce  the localized algebraic index of the complete symbol of an elliptic operator.
With the help of a calculus of semiclassical quantized canonical transformations, a version of Egorov's theorem and a theorem on trace asymptotics for 
semiclassical Fourier integral operators we show that the localized analytic index and the localized algebraic index coincide.
As a corollary, we express the Fredholm index in terms of  the algebraic index.

{\bf Acknowledgment.} The support through grants DFG SCHR 319/8-1  and RFBR 19-01-00574 is gratefully acknowledged.

\section{Quantized Canonical Transformations}
Let $M$ be a closed manifold of dimension $n$. Its cotangent bundle $T^*M$ carries a natural symplectic structure, 
given in local coordinates by the form 
$$\go = \sum_{j=1}^n dx^j \wedge d\xi^j.$$
We let 
$T_0^*M= T^*M\setminus 0$
be the cotangent bundle with the zero section removed. 
Recall that a symplectomorphism  is a diffeomorphism 
$$C: T_0^*M\rightarrow T_0^*M$$ 
preserving the symplectic form: $\go(C(x,\xi),C(y,\eta)) = \go((x,\xi), (y,\eta))$.   

\begin{definition}\rm
A {\em canonical transformation} is a symplectomorphism which in addition is positively homogeneous of degree $1$ in the fiber: 
$C(y,\eta)=(x,\xi)$ implies that $C(y,\gl\eta)=(x,\gl\xi)$ for all $\gl>0$.    
\end{definition} 

\begin{example}\label{ex.1}
{\rm (a)} 
Let $\ga: M\to M$ be a diffeomorphism of $M$. Then a canonical transformation 
$C_\ga$ is defined by 
$$C_\ga(y,\eta) = (\ga^{-1}(y), {}^t(\partial \ga(y))^{-1}\eta).$$
Symplectomorphisms of this type extend to the full cotangent bundle. 

{\rm (b)} Let $H=H(x,\xi)$ be a smooth function on $T^*_0M$ which is positively homogeneous of degree $1$ in the fiber. 
We denote by $t\mapsto F_t$, $t\in \R$,  the flow on $T^*_0M$ generated by the Hamiltonian vector field  $V_H$  induced by $H\ ($recall that $V_H$ is defined by the relation $\iota_{V_H}\go = dH)$.

Then the map $(y,\eta) \mapsto F_t(y,\eta)$ defines a canonical transformation  for every $t\in \R$.   
\end{example}

In this note, which is based on the articles \cite{SS1} and \cite{SSS2}, we will study {\em quantized canonical transformations}, i.e., bounded operators on $L^2(M)$ associated with a canonical transformation in a sense we shall explain next. We start with the following simple fact:  

\begin{lemma}
The $($twisted$)$ graph of a canonical relation defines a Lagrangian submanifold of $T^*_0M\times T^*_0M$. More precisely: 
The set  
$$\Lambda = \{((x,\xi),(y,-\eta))\in T^*_0M\times T^*_0M:  (x,\xi) = C(y,\eta)\}$$
is a Lagrangian submanifold of $T^*_0M\times T^*_0M$, endowed with the symplectic form $\go\oplus \go$. 
\end{lemma}  

The word `twisted' refers to the sign change in the second variable. The next observation is that, locally, a Lagrangian submanifold $\gL$ of $T^*_0M\times T^*_0M$ can be described by a phase function, see \cite[Theorem 21.2.16]{HIII} for a proof  

\begin{theorem}\label{T.1}
 Let $((x_0,\xi_0),(y_0,\eta_0))\in \gL$.  
 Then there exist neighborhoods $U$ of $x_0$,
 $V$ of $y_0$, an open  cone $\gG$ in $\R^d$ $($for suitable $d)$ and a function  $\varphi: U\times V\times \gG\to \R$, homogeneous of degree $1$ in $\gt$ and non-degenerate in the sense that the differentials $d(\partial_{\theta_j}\varphi))$, $j=1,\ldots,d$, are linearly independent, such that locally near 
$((x_0,\xi_0),(y_0,\eta_0))$ the set $\gL$ is given as 
$$%\gL \stackrel{\text{loc.}}= 
\{((x,\partial_x\gvp(x,y,\gt)), (y,-\partial_y\gvp(x,y,\theta))): \partial_\gt\varphi(x,y,\theta) = 0\}.$$
\end{theorem}

In other words, the map $\ga: \crit_\gvp\to \graph C$, given by 
\begin{eqnarray}\label{alpha}%\lefteqn{}\\
\ga(x,y,\theta) = ((x,\partial_x\gvp(x,y,\theta)),(y,-\partial_y\gvp (x,y,\gt)) )
\end{eqnarray}
is a diffeomorphism from the critical set
$$\crit_\varphi= \{(x,y,\theta): \partial_\gt(x,y,\theta) =0\},$$  
onto a conical neighborhood of $((x_0,\xi_0),(y_0,\eta_0))$. 

A canonical transformation therefore locally is given by a non-degenerate phase function and conversely, a non-degenerate phase function determines (locally) 
a Lagrangian submanifold in $T_0^*M\times T^*_0M$. 

A bounded linear operator $\Phi: L^2(M) \to L^2(M)$ is a quantized canonical transformation, if it is microlocally given by oscillatory integral kernels determined by an amplitude function $a$ in a suitable Hörmander class $S^m_{\rm cl}(\R^n\times \R^n\times \R^d)$ and a phase function $\varphi$ associated with a canonical transformation $C$. In particular, the quantized canonical transformations form a subclass of the Fourier Integral Operators on $M$  

In order to be more precise, let $((x_0,\xi_0), (y_0,\eta_0)) \in T^*_0M\times T_0^*M$ with 
$(x_0,\xi_0) = C(y_0,\eta_0)$ and let, with the notation of Theorem \ref{T.1},  
$\varphi: U\times V\times \gG\to \R$ be a phase function locally 
describing the twisted graph of $C$. 

\begin{definition} 
We call $\Phi$ a {\em quantized canonical transformation}, if,
 in a conic neighborhood of $((x_0,\xi_0),(y_0,\eta_0))$, 
 the Schwartz kernel $K^\Phi$ of $\Phi$ can be written 
\begin{eqnarray}\label{2.5.1}%\lefteqn{}\\
K^\Phi (x,y) = %(2\pi)^{-n} 
\int e^{i\varphi(x,y,\theta)} b(x,y,\theta)  \, d\gt  +K
\end{eqnarray} 
where $b$ in $S^{(n-d)/2}_{\rm cl}$ vanishes near $\gt=0$ and outside $U\times V\times \Gamma$ and $K$ is a $C^\infty$ kernel function. 
\end{definition} 
The full Schwartz kernel is then obtained with the help of a partition of unity. 
The choice of the order $(n-d)/2$ makes the induced operator continuous on 
$L^2(M)$. 

\begin{example}\label{ex.3}
Let $\ga: M\to M$ be a diffeomorphism. 
Then the `shift operator' $T_\ga\in \sL(L^2(M))$ given by 
$T_\ga u(x) = u(\ga^{-1}(x))$ is a quantized canonical transformation 
associated with the canonical transformation $C_\ga$ in Example {\rm \ref{ex.1}}. 

Indeed, for $u\in \sS(\R^n)$ we have 
\begin{eqnarray*}u(\ga^{-1}(x)) &=& (2\pi)^{-n} \int e^{i\skp{\ga^{-1}(x),\gt}} \hat u (\gt)\,d\gt 
=
(2\pi)^{-n} \int e^{i\skp{\ga^{-1}(x)-y,\gt}}  u (y)\,dy d\gt
\end{eqnarray*}
Hence, the Schwartz kernel has the (even globally defined) phase $\gvp(x,y,\gt)= \skp{\alpha^{-1}(x)-y,\gt}$. Applying the local formula for the associated Lagrangian submanifold in Theorem {\rm\ref{T.1}} yields the graph of $C_\ga$:
\begin{eqnarray*}\lefteqn{
\{((x, {}^t\partial_x \ga^{-1}(x)\gt),( y,\gt )): y=\ga^{-1}(x)\}}\\
&=&\{((\ga(y), {}^t\partial \ga^{-1}(\ga(y)) \gt), (y,\gt))\}
=\{((\ga(y), {}^t\partial \ga(y)^{-1} \gt), (y,\gt ))\}.
\end{eqnarray*}
\end{example}

\begin{example}
Consider the operator $e^{it\sqrt{\Delta}}$, for simplicity on $L^2(\R^n)$. 
Here, 
$$e^{it\sqrt\Delta} u(x) = (2\pi)^{-n} \int e^{i\skp{x-y,\gt}} e^{it|\gt|}u(y)\, dy d\gt, \quad u\in \sS(\R^n).$$
The phase function is $\gvp(x,y,\gt) = \skp{x-y,\gt} + t|\gt|$, parametrizing {\rm(}in
the sense of Theorem {\rm\ref{T.1}}{\rm)} the set 
$$\{((x,\gt)),(y,\gt)): x-y +t\gt/|\gt|=0\} = \{((y-t\gt/|\gt|,\gt),(y,\gt))\}.$$
This is the graph of the canonical transformation induced by the flow of the Hamiltonian
vector field associated to the function $H(\gt)=|\gt|$ on $T^*M $ in the sense 
of Example {\rm\ref{ex.1}(}b{\rm)}.   
\end{example}

\section{Operator Algebras} 
\subsection*{Operators} 
Let $G$ be a discrete group with unit element $e$ and let $\Phi_g$, $g\in G$, be quantized canonical transformations in $\sL(L^2(M))$ satisfying 
\begin{eqnarray}\label{e2.1}%\lefteqn{}\\
\Phi_e \equiv I  \text{ and  } \Phi_g\Phi_h \equiv \Phi_{gh} \text{ modulo }
\sK(L^2(M)).
\end{eqnarray}
Write $C_g$ for the canonical transformation associated with $\Phi_g$. 

We shall consider operators of the form 
\begin{eqnarray}\label{e2.2}%\lefteqn{}\\
D= \sum_{g\in  G} D_g\Phi_g +K,
\end{eqnarray} 
where the $D_g$ are classical pseudodifferential operators of order zero, only finitely many are different from zero, and $K\in \sK(L^2(M))$. 

The following lemma collects a few basic properties of these operators. 
\begin{lemma}\label{l3.1}
{\rm (a)}
The operators $\Phi_g$ are all Fredholm operators, and  $\Phi_{g^{-1}}$ furnishes a Fredholm inverse to $\Phi_g$. 

{\rm(b)} The operators of the form \eqref{e2.2} form an algebra. 
\end{lemma}

\begin{proof} (a) is an immediate consequence of Condition \eqref{e2.1}. 

In order to establish (b), it suffices to consider a product $(D_g\Phi_g)(D_h\Phi_h)$ for $g,h\in G$. We recall Egorov's Theorem: Given a pseudodifferential operator $A$ of order zero we have
$$\Phi_g A \Phi_{g^{-1}}= \tilde A,$$ 
where $\tilde A$ also is pseudodifferential of order zero and the principal symbols satisfy the relation
$$\sigma_{\tilde A}(x,\xi)  = \sigma_A (C_{g^{-1}}(x,\xi)).$$
Hence, modulo $\sK(L^2(M))$,
$$D_g\Phi_g D_h\Phi_h\equiv D_g \Phi_gD_h\Phi_{g^{-1}}\Phi_g\Phi_h
\equiv D_g \tilde D_h\Phi_{gh}$$
is of the required form. 
\end{proof} 
 
\begin{lemma}\label{l3.2}
We may assume that the operators $\Phi_g$ are unitary 
modulo $\sK(L^2(M))$, i.e. we may assume that, in addition to the properties \eqref{e2.1},
$$(\Phi_g)^*\Phi_g \equiv I \equiv \Phi_g (\Phi_g)^*\text{ modulo } \sK(L^2(M)).$$
\end{lemma}

\begin{proof}By Egorov's theorem $\Phi_g(\Phi_g)^*$ is a nonnegative pseudodifferential operator. Somewhat informally we denote by
$\tilde \Phi_g=(\Phi_g(\Phi_g)^*)^{-1/2}\Phi_g$ the unitary part in the polar decomposition of $\Phi_g$. A computation as in the proof of Lemma \ref{l3.1}
shows the assertion.
\end{proof}

In the sequel we shall therefore assume that the $\Phi_g$ form an almost unitary 
representation as in Lemma \ref{l3.2}. 

\subsection*{Alternative: Finite-dimensional Lie groups}
Given a finite-dimensional Lie group $G$ and a unitary representation 
$$\rho: G\to \sL(L^2(M)), \qquad g\mapsto \Phi_g$$ 
of $G$ by quantized canonical transformations $\Phi_g$, we may also consider 
(under certain technical assumptions) operators of the form 
\begin{eqnarray*}%\lefteqn{}\\
D= \int_G D_g \Phi_g \, d\mu(g): L^2(M) \to L^2(M),
\end{eqnarray*}
where $g\mapsto D_g$ is a smooth, compactly supported family of pseudodifferential operators of order zero and $d\mu(g)$ denotes an invariant measure on $G$. 
As operators of this form are rarely Fredholm, one then studies the 
index theory of operators of the form $I+D$. 
We shall not consider these operators in this note; details can be found in \cite{SSS2}.

\subsection*{Special cases} 
1. If the group is $\{e\}$, then $D$ is just a classical pseudodifferential operator,
and we obtain the well-known index problem of Atiyah and Singer.  
 
 2. If $D=\Phi_g$ for a single quantized canonical transformation, then determining the index is known as the Atiyah-Weinstein index problem. It has
been solved independently and by different methods by Epstein and Melrose \cite{EM98} for the case of canonical transformations on $T^*_0M$ and,
more generally, for canonical transformations between possibly different 
manifolds, by Leicht\-nam, Nest and Tsygan \cite{LNT01}. 

3. If the $\Phi_g$ are associated with the special canonical transformations 
in Example \ref{ex.1}(a), then the $\Phi_g$ are shift operators as pointed 
out in Example \ref{ex.3}. The corresponding algebra of operators has been 
studied by Antonevich and Lebedev e.g. in \cite{AnLe1}, Connes and Moscovici \cite{CoMo2}, Perrot \cite{Per5} as well as by Savin, Schrohe and Sternin \cite{SaScSt}, \cite{SaSt22},  \cite{Ster20}.  

4. The index problem considered by B\"ar and Strohmaier in \cite{BaSt1}
is closely related as it reduces to a Toeplitz variant of the problems considered here.  

5. If $G$ is a  Lie group acting on $M$ locally-freely, then the
$G$-operators coincide with the transverse pseudodifferential operators   
with respect to the foliation on $M$ defined by the orbits of the group action which were studied by Kordyukov in \cite{Kord3}.

More examples can be found in \cite{SSS2}.  

\subsection*{Symbols and ellipticity}
Recall that a  $C^*$-dynamical system is a triple $(A,G,\gb)$, 
consisting of a $C^*$- algebra $A$, a locally compact group $G$ and a strongly continuous homomorphism $\gb : G \to \operatorname{Aut} A$ (i.e.,
$g\mapsto \gb_gh$ is continuous for all  $h\in A$).

A covariant representation of a dynamical system  $(A, G, \gb)$ is a pair $(\pi, U )$, 
consisting of a representation $\pi : A \to \sL(H)$ and a unitary representation 
$U : G \to \sU(H)$ on the same Hilbert space such that
$\pi (\gb_g(a)) = U_g\pi (a)U_g^*$.

A pair $(\pi,U)$ induces a representation 
$\pi \rtimes U: C_c(G,A)\to \sL(H) $ by
$$(\pi\rtimes U)(f) = \int _G \pi(f(g))U_g\, d\mu(g).$$
We can define a norm on $C_c(G,A)$ by letting 
$$\|f\| = \sup\{\|(\pi\rtimes U)f\|: (\pi,U) \text{ covariant representation}\}.$$
The sup is finite, since $\|(\pi\rtimes U)f\|\le \|f\|_{L^1}$. 

\begin{definition}
The $($maximal$)$ crossed product $A\rtimes_\gb G$ is the closure of $C_c(G,A)$ in this norm. 
\end{definition} 

In the case at hand, we have (possibly after the modification in Lemma \ref{l3.2}) 
an almost unitary representation  
$\Phi_g: G\to \sL(L^2(M))$, i.e., a unitary representation with values in the 
Calkin algebra $\sL(L^2(M))/\sK(L^2(M))$. 
Since the latter is a $C^*$-algebra, it is a 
subalgebra of $\sL( H)$ for some Hilbert space $ H$. Denote by $\iota: 
\sL(L^2(M))/\sK(L^2(M))\to \sL(H)$ the embedding and by $q: \sL(L^2(M))\to 
\sL(L^2(M))/\sK(L^2(M))$ the canonical projection. 
Note that we have a strongly continuous action $\gb$ of  $G$ on the $C^*$-algebra $A= C(S^*M)$ via the canonical transformations: 
$$\gb(a) = C_g^{-1*}a$$
This yields a covariant representation  $(\pi,U)$ of the $C^*$-dynamical system $(C(S^*M), G, \gb)$ via 
\begin{eqnarray*}%\lefteqn{}\\
\pi(a) &=& \iota(q(A)) \in \sL(H)\\
U_g &=& \iota(\Phi_g) \in \sU(H),
\end{eqnarray*}
where $A$ is any operator in the closure $\overline \Psi$ of the algebra of 
zero order pseudodifferential operators in $\sL(L^2(M))$ with symbol $a$.
Indeed, this is a consequence of Egorov's theorem, since 
$$U_g \iota (a) U_g^{-1} = \iota (C_g^{-1*}a).$$

\begin{definition}\label{D.ell}
The symbol associated to an operator $D$ as in \eqref{e2.2} is the collection 
$$\gs(D) = \{\gs(D_g)\}_{g\in G}$$ 
of the principal symbols 
of the pseudodifferential operators $D_g$. We consider $\gs(D)$ as an 
element of the crossed product algebra $C(S^*M) \rtimes G$.   

We call $D$ elliptic, if $\gs(D)$ is invertible in  $C(S^*M) \rtimes G$. 
\end{definition}

Note that we fix here the representation of  $D$ with the particular choice of  $D_g$'s  and $\Phi_g$'s.   

%Conversely, given a finite collection of symbols $\{D_g\}$, we can associate 
%with it the operator $D = \sum_g D_g \Phi_g$. 
%More is true: 
The considerations before Definition \ref{D.ell} show that the quantization map 
$$C_c(G, C(S^*M))\ni \{\gs(D_g)\}_{g\in G} \mapsto \sum_{g\in G}\iota(\gs(D_g))\iota (\Phi_g)\in \sL(H)$$
naturally extends to a map 
%from a map $C_c(G, C(S^*M))\to \sL(L^2(M))$ to a map 
%from the maximal crossed product 
$$Q: C(S^*M)\rtimes G\to \sL(H).$$ 

\begin{theorem}
If $D$ is elliptic, then $D: L^2(M) \to L^2(M)$ is a Fredholm operator. 
\end{theorem} 

\begin{proof} 
Given an elliptic operator $D$, the inverse $\gs(D)^{-1}$ exists in $C(S^*M)\rtimes G$.  Then $E=Q(\gs(D)^{-1}) $ furnishes an element of $\iota(\sL(L^2(M))/\sK(L^2(M))$. Any representative in $\sL(L^2(M))$ is a Fredholm inverse to $D$.
\end{proof}

\section{Analytic and Algebraic Indices}
The next task is to determine the Fredholm index of elliptic elements. 
Our long term goal is to do this using algebraic index theory in the spirit of Fedosov \cite{Fds7} and Nest and Tsygan \cite{NeTs1}. 
In  a first step we outline here how to establish the equality of analytic 
and algebraic indices. In fact, we  define localized versions of these indices and 
show their equality. 
To this end we specialize further and assume that $G$ is a finite extension of 
$\Z^d$. We denote by $e$ the unit element and by $\skp g$ the conjugacy class 
of $g\in G$. We furthermore suppose  that the operators $\Phi_g$, $g\in G$, 
are unitary and satisfy 
$$\Phi_e=I,\qquad \Phi_g\Phi_h=\Phi_{gh}.$$ 
We shall study an elliptic operator 
\begin{eqnarray*}%\label{}\lefteqn{}\\
D=\sum_{g\in G} D_g\Phi_g
\end{eqnarray*}
as in \eqref{e2.2}. 
We make the additional assumption that the inverse to the 
symbol $\gs(D)$ is an element of the algebraic crossed product  
\begin{eqnarray}%\label{}\lefteqn{}\\
\gs(D)^{-1} = \{r_g\}_{g\in G} \in C^\infty (S^*M)\rtimes_{\rm alg} G.
\end{eqnarray}

\subsection*{The localized analytic index}

\begin{definition}
We introduce the algebraic crossed product $\Psi^{-N} (M)\rtimes_{\rm alg} G$ 
where $\Psi^{-N}$, $N\in  \N_0$,  is the algebra of classical pseudodifferential operators of order $-N$ 
and $G$ acts on $\Psi^{-N}$ by conjugation: $A\mapsto \Phi_g A\Phi_g^{-1}$. 
\end{definition} 

\begin{lemma}
For each $N\ge 1$ there exists an operator $E \in\Psi^0(M)\rtimes_{\rm alg} G$ such that  
\begin{equation}\label{eq-ai1}
I-DE ,I-E D\in \Psi^{-N}(M)\rtimes G.
\end{equation}
\end{lemma}  

\begin{proof}
Let 
$
 E_0=\sum_g R_g\Phi_g$, where the $\sigma(R_g)=r_g$. 
Then $S_1= I-E_0D$ and $S_2=I-DE_0$ are elements of $\Psi^{-1}(M)\rtimes G$.
We obtain the assertion by setting 
$ E = (I+S_1+\ldots+S_1^{N-1}) E_0$.
\end{proof}

The elements of $\Psi^{-N}(M)$ are of trace class whenever $N>\dim M$. 
This motivates the following definition. 
  
\begin{definition}
Given $g\in G$ and an element $\sum_l K_l \Phi_l\in \Psi^{-N}\rtimes_{\rm alg } G$, $N>\dim M$, we introduce  the trace functional 
\begin{equation}    
\Tr_g\Big(\sum_l K_l\Phi_l\Big)=   \sum\limits_{l\in\langle g\rangle}\tr\left(K_l\Phi_l\right),
\end{equation}
with the operator trace $\tr$ on $L^2(M)$ and define the  analytic index of $D$ localized at the conjugacy class $\skp g$ by 
$$
 \ind_g D=\Tr_g(1-E  D)-\Tr_g(1-DE )=\Tr_g[D,E]\in \mathbb{C},
$$
where $E$ is an inverse modulo $\Psi^{-N} \rtimes_{\rm alg} G$. 
\end{definition} 
 
 The following is Proposition 19 in \cite{SS1}:
\begin{lemma}
{\rm(a)} The localized index $\ind_g D$ is independent of the choice of the almost-inverse operator and therefore a well-defined   
 invariant of the complete symbol of $D$.
 
{\rm (b)} The Fredholm index of $D$ is given by 
$$
  \ind D=\sum_{\langle g\rangle\subset G}\ind_g D,
 $$  
 where the sum is over all conjugacy classes in  $G$;

%{\rm (c)} Let 
%\begin{equation}\label{proj1}
%   W_D=\left(%
%          \begin{array}{cc}
%             (2-DR)D  & 1-DR \\
%             RD-1 & R \\
%          \end{array}%
%       \right)
%,\qquad  P_0=\left(
%           \begin{array}{cc}
%             1  & 0 \\
%             0  & 0          
%           \end{array}
%         \right).
% \end{equation}
% Then 
%  \begin{equation}\label{proj0}
%  \ind_g D= \Tr_g(W_DP_0W_D^{-1}-P_0),
% \end{equation}

\end{lemma}

\subsection*{Semiclassical calculus} 
In order to define the algebraic index of the operator $D$, we introduce a 
semiclassical calculus for operators of this type. 
In fact, given $h>0, \gve>0$ and $N\in \N_0$ we define from the operator $\Phi_g$  operators $\Phi_{g,h,\gve,N}$ as follows. 

 Given a point $((x_0,\xi_0), (y_0,\eta_0))$ in the graph of $C_g$ in $T^*_0M\times T^*_0M$ 
 suppose that the integral 
 kernel $K^{\Phi_g} $ of $\Phi_g$ is locally given, modulo smooth kernels, 
 by an oscillatory integral as in \eqref{2.5.1}.
Let 
$$
b(x,x',\theta)\sim \sum_{j\ge 0} b_j(x,x',\theta)  \text{ as $|\theta|\to \infty$},
$$
with $b_j(x,x',\theta)$ homogeneous of degree $(n-d)/2-j$ in $\theta$ be the 
asymptotic expansion of $b$. 
 
We define the semiclassical Fourier integral operator $\Phi_{h,\varepsilon,N}$ associated with $\Phi$ as the operator with the integral  kernel
\begin{equation}\label{ker2}
K^\Phi_{h,\varepsilon,N}(x,y) = h^{-d/2-n/2}\int e^{\frac i h \varphi(x,y,\theta)} \sum_{0\le j< N}h^j b_j(x,y,\theta) \chi (x,y,\theta) d\theta,
\end{equation}
where the smooth function $\chi$ is chosen such that, with 
$\alpha$ defined in \eqref{alpha},
\begin{eqnarray*}\label{chi}
%\mbox{\ \ \ \ \ \ \ } 
\chi(x,y,\theta)=
 \left\{
  \begin{array}{ll}
   1 & \text{in an open neighborhood of  }\alpha^{-1}(T^*_0M\times\{|\eta|\ge\varepsilon\}),\\ 
  & \text{which is conic at infinity,} \vspace{2mm}\\
   0 & \text{in a small neighborhood of the zero section}.   
  \end{array}
 \right.
\end{eqnarray*}
Here, $|\eta|$ is defined via a choice of a Riemannian metric on $M$, 
and a subset $U$ of $M\times M\times \R^d$  is called conic at infinity, if there exists an $R>0$ such that 
$(x,y,\gl \gt)\in U$ whenever $|\gt|\ge R$, $\gl\ge1$ and $(x,y,\gt)\in U$.

We introduce  the semiclassical Sobolev spaces:
\begin{definition}
\label{Hsh}
{\rm (a)} 
The space $H_h^s(M)$ is the set of all distributions $u$ on $M$
such that
$$\|u\|_{H^s_h}= \|(h^2\Delta +1)^{s/2}u\|_{L^2}<\infty,$$ 
where $\Delta$ stands for the nonnegative Laplacian on $M$.

{\rm (b)} An $O(h^N)$-operator family is a family of operators of order  $-N$
whose norm in $\sL(H_h^s(M), H_h^{s+N}(M))$  is of the order $O(h^N)$ as $h\to 0$ for every $s$.
\end{definition}

A  semiclassical symbol $a=a(x,\xi,h)$  of order $m$ in  a chart in $T^*M$ with coordinates $(x,\xi)$ is a smooth family of symbols  
with parameter $h>0$ and an asymptotic expansion 
\begin{equation*}\label{eq-2}
a(x,\xi,h)\sim \sum_{j\ge 0} h^j a_j(x,\xi), \quad \text{as $h\to 0$},
\end{equation*}
where $a_j(x,\xi)\in S^{m-j}$, i.e., 
for all $N\ge 0$ 
$$
 h^{-N}\Bigl(a(x,\xi,h)- \sum_{0\le j\le N} h^j a_j(x,\xi)\Bigr) \longrightarrow 0 \quad \text{in }S^{m-N}\text{ as }h\to 0.
$$
We call $a_0\in S^{m}$ the {\em leading symbol} of $a$.

A semiclassical symbol $a$ defines a pseudodifferential operator $\op_h(a)$. In 
local coordinates
$$
 \op_h(a)u(x)=\frac{1}{(2\pi h)^n} \iint e^{\frac{i}{h}(x-y)\xi}a(x,\xi,h)u(y)dyd\xi.
$$

We then obtain the following result whose proof occupies a large part of 
\cite{SS1}: 

\begin{proposition}\label{fio1}
\begin{enumerate}\renewcommand{\labelenumi}{(\alph{enumi})}
\item {\rm (Correctness of the definition)}  The operator family $\Phi_{h,\varepsilon,N}$ with integral kernel \eqref{ker2} is   independent of the choice of the representation \eqref{2.5.1} and the function $\chi$  modulo sums of $O(h^N)$-families  and families, which become $O(h^\infty)$-families when composed to the right by $\op_h(a)$ where $a(y,\eta)$ vanishes for $|\eta|<2\gve$. 
%with oscillatory front of the integral kernel contained in the set 
% $(\graph C) \cap \{(x,\xi);(y,\eta)): |\eta|\le \varepsilon\}$.
 
 \item {\rm (Composition formula)} Given quantized canonical transformations $\Phi',\Phi''$ associated with $C'$ and $C''$ and a  semiclassical symbol $a$, we find that for  $\Phi=\Phi'\Phi''$
\begin{equation*}\label{eq-ff1}
 \Phi_{h,\varepsilon,N}\op_h(a)=\Phi'_{h,\varepsilon,N} \Phi_{h,\varepsilon,N}''\op_h(a) \mod O(h^N),
\end{equation*} 
provided that $\varepsilon>0$ is chosen such that
$a$ vanishes on the subsets $\{|\xi|<\varepsilon\},{C''}^{-1}\{|\xi|<\varepsilon\}\subset T^*_0M$. 
A similar statement holds, if we take the product with $\op_h(a)$ on the left and choose $\varepsilon$ appropriately. 

 \item {\rm (Semiclassical Egorov theorem)}
Given a semiclassical symbol $a$, which   vanishes on the sets $\{|\xi|<\varepsilon\}, C^{-1}\{|\xi|<\varepsilon\}\subset T^*M$,
the composition
\begin{equation*}\label{op-3}
  \Phi_{h,\varepsilon,N} \op_h(b) \Phi^{-1}_{h,\varepsilon,N},
\end{equation*}
where $\Phi^{-1}_{h,\varepsilon,N}$ is associated with $\Phi^{-1},$
is a semiclassical pseudodifferential operator with symbol equal to
\begin{eqnarray*}%\label{}
{
\sigma(\Phi_{h,\varepsilon,N} \op_h(a) \Phi^{-1}_{h,\varepsilon,N})}
=\Bigl(1+\sum_{1\le k\le N, 0<|\alpha|+|\beta|\le 2k}h^k \mu_{k,\alpha,\beta} D^\alpha_x D^\beta_\xi \Bigr)(C^{-1})^*a 
\end{eqnarray*}
modulo symbols  inducing $O(h^N )$-families.
%Here $\varepsilon$ is chosen such that $b$. Moreover, t
The coefficients $\mu_{k,\alpha,\beta}(x,\xi)$ are homogeneous functions in $\xi$ of  degree $|\beta|-k$,
and are expressed in terms
of the amplitudes and phase functions of $\Phi$ and $\Phi^{-1}$. They do not depend on the choice of the cut-off functions or  $\varepsilon$. 
\end{enumerate}
\end{proposition}
 
\subsection*{The localized algebraic index}
We denote by $\mathbb A$ the algebra of all zero order semiclassical pseudodifferential operators whose symbols vanish in a neighborhood of the zero section. 
According to Proposition {\rm\ref{fio1}(c)} the mapping
$$\A\times G \ni (a,g) \mapsto \sigma(\Phi_{h,\varepsilon,N} \op_h(a) \Phi^{-1}_{h,\varepsilon,N})$$  
{\rm(}for sufficiently small $\gve>0)$ defines an action of $G$ on $\A$. 
We can therefore define the algebraic crossed product $\A \rtimes_{\rm alg}G$.  We denote the product in this algebra by $*$. 
 
Let $\B=(\A\rtimes_{\rm alg}G)^+$ be the algebra $\A \rtimes_{\rm alg}G$ with a unit adjoined. Its
elements are given as collections
\begin{equation}\label{coll1}
    \Big\{\sum_{j\ge 0}h^j a_{l,j}(x,\xi)\Big\}_{l\in G},
\end{equation}
where $a_{l,j}(x,\xi)\in \A$ for $l\not=e$ while  $a_{e,0}(x,\xi)$ is allowed 
to be equal to a nonzero constant in a neighborhood of the zero section in $T^*M$. Denote by $\mathbb{B}_{N}\subset \mathbb{B}$
the ideal of elements \eqref{coll1} with  coefficients of order 
$\le -N$.

We call  $a\in \B$ elliptic, if its leading (semiclassical) symbol 
$a_0\in S^0\rtimes_{\rm alg}G$ is 
invertible modulo symbols of order $-1$. 

\begin{lemma}
Let   $a\in \mathbb{B}$ be an elliptic symbol of order zero. Then for each $N\ge 1$ there exists a symbol $r_N \in \mathbb{B}$ such that  
\begin{equation}\label{eq-ai2}
 1-a*r_N,1-r_N*a  \in \mathbb{B}_{N}.
\end{equation} 
\end{lemma}

\begin{proof}
Since $a$ is elliptic, there exists $r_0\in S^{0}(T^*_0M)\rtimes G$ such that $1-a_0*r_0$ is of order $-1$. Then, modulo $\B_{N+1}$, 
%Clearly, $r_0\in \mathbb{B}$ and we have
\begin{equation*}
 a*r_0\equiv\Big(a_0+\sum_{1\le j\le N}h^ja_j\Big)*r_0=a_0*r_0+\sum_{1\le j\le N}h^ja_j*r_0 \equiv 1-w, 
\end{equation*}
where  
$w\equiv(1-a_0*r_0)-\sum_{1\le j\le N}h^ja_j*r_0 \in \mathbb{B}_{1}$.
Hence
$$
a*r_0*(1+w+w*w+...+w^N)\equiv(1-w)*(1+w+w*w+...+w^N)=1-w^{N+1},
$$
with $w^{N+1}\in \mathbb{B}_{N+1}$m and $ r=r_0*(1+w+w*w+...+w^N)$ 
furnishes the desired right inverse. A computation shows  that also $r*a-1\in \mathbb{B}_{N+1}$. 
\end{proof}

A crucial fact now is the following theorem, shown in \cite{SS1}.

\begin{theorem}
Let $a\in \A$ be of order $<-2\dim M$ and let $h\in G$ be associated with the canonical transformation $C$. Suppose that $C$ is of finite order, i.e. $C^k=I$ for some $k$, and denote by $T^*M^C$ the fixed point set of $C$. Then the operator 
$\op_h(a) \Phi_{h,\gve,N}$ is of trace class for every $h$ and we have an asymptotic expansion 
\begin{equation}\label{eq-5rr}
 \tr(\op_h(a)\Phi_{h,\varepsilon,N})\sim h^{-\dim  T^*M^C/2}\sum_{j\ge 0} \alpha_jh^j
 %, \quad  \alpha_j=\int_{T^*M^C}m_j
\end{equation}
in integer powers of $h$ $($the fixed point sets $T^*M^C$ are even-dimensional by  \cite{Fds15}$)$.  
The coefficients $\alpha_j$ in \eqref{eq-5rr} do not depend on the choices in the construction  of $\Phi_{h,\varepsilon,N}$ up to $j=N-1$.
\end{theorem} 

\begin{definition}
Let $a=\{a_l\}_{l\in G}\in \A\rtimes_{\rm alg} G$ be of order $<-2\dim M$. 
Define the trace functional $\tau_{g,N}$ localized at $g\in G$ by 
\begin{equation}\label{eq-6}
   \tau_{g,N}(a)= \sum\limits_{l\in\langle g\rangle}  \tr(Op_h(a_{l})\Phi_{l,h,\varepsilon,N})\in 
   (h^{-\dim T^*M^g/2} \mathbb{C}[h])/h^{N-\dim M}, 
\end{equation}
where $(h^{-\dim T^*M^g/2} \mathbb{C}[h])/h^{N-\dim M}$ stands for the space of  Laurent polynomials 
$$
\sum_{-\dim T^*M^g/2\le j<N-\dim M}c_j h^j
$$
and $\varepsilon$ in the definition of the $\Phi_{l,h,\varepsilon,N}$  is chosen such  that the first $N$ components in the expansion of $a_{l}\in \mathbb{A}$
in powers of $h$ are equal to zero on the set 
$C_l\{|\xi|<2\varepsilon\}\subset T^*M$.
\end{definition} 

It turns out that the definition is independent of the choices involved in the definition of the $\Phi_{l,h,\gve,N}$. Moreover, $\tau_{g,N}$ is a trace.

\begin{definition}\label{algind1}
Given an  elliptic symbol $a\in \mathbb{B}$, its  algebraic index localized at the conjugacy
class $\langle g\rangle\subset G$  is defined as
\begin{eqnarray}\label{eq-algind4}
 \widetilde{\ind}_{g,N} a&=&\tau_{g}(1-r_N*a)-\tau_{g}(1-a*r_N)=\tau_{g}[a,r_N]\\
 &&\in \left(h^{-\dim T^*M^g/2} \mathbb{C}[h]\right)/h^{N-\dim M} ,\nonumber
\end{eqnarray}
where $r$ is an almost-inverse symbol for $a$ such that \eqref{eq-ai2} holds.
\end{definition}
The algebraic index \eqref{eq-algind4} is independent of the choice of the almost-inverse symbol $r_N$
and the algebraic indices for different $N$ are compatible:
$$
 \widetilde{\ind}_{g,N} a\equiv \widetilde{\ind}_{g,N+1}a \mod h^{N-\dim M}.
$$
They  define the algebraic index as $N\to \infty$
\begin{equation}\label{eq-algind4a}
 \widetilde{\ind}_{g} a\in h^{-\dim T^*M^g/2} \mathbb{C}[[h]].
\end{equation}

The main result of \cite{SS1} is:
\begin{theorem}\label{th-1}
Given a finite order element $g\in G$, the algebraic index localized at $g$ has no negative and no positive powers of $h$, and its constant term is equal to the analytic index:
\begin{equation}\label{eq-7}
  \ind_g \op_h(a)=(\widetilde{\ind}_g a)_{|h=0}.
\end{equation}
\end{theorem}

What if $g$ is not of finite order? The following proposition gives a partial answer:

\begin{proposition}\label{prop7}
Suppose that there exists a group homomorphism $\chi:G\to \mathbb{Z}$  such that $\chi(g_0)\ne 0.$
 Given an elliptic operator $D$ and $g_0\in G$, we then have
 $$
  \ind_{g_0} D=0.
 $$
\end{proposition}  

This condition is satisfied for all elements of infinite order, provided the group is a 
finite extension of $\Z^d$, $d\in \N_0$. Hence we obtain: 

\begin{corollary}
Suppose that $G$ is a finite extension of $\Z^d$. 
Given an elliptic symbol $a\in \B$, the Fredholm index of the corresponding operator  $A$
 is equal to the sum of localized algebraic indices over the torsion conjugacy classes in $G$:
\begin{equation}\label{eq83}
\ind A= \sum_{\langle g\rangle\subset {\rm Tor}\;G} (\widetilde{\ind}_g a)_{|h=0}.
\end{equation}
Here ${\rm Tor}\;G$ is the torsion subgroup of $G$.
\end{corollary}

\newpage
\begin{flushleft}
Anton Savin\\
Peoples' Friendship University of Russia (RUDN University) \\
6 Miklukho-Maklaya St\\ Moscow, 117198\\ Russia \\
E-mail address: \texttt{antonsavin@mail.ru}
\end{flushleft}
\begin{flushleft}
Elmar Schrohe \\
Institute of Analysis\\ Leibniz Universität Hannover\\ Welfengarten 1\\ 
30167 Hannover\\ Germany\\
E-mail address: \texttt{schrohe@math.uni-hannover.de}
\end{flushleft}

\end{document}